\documentclass[12pt,reqno]{article}

\usepackage[usenames]{color}
\usepackage{amssymb}
\usepackage{graphicx}
\usepackage{amscd}

\usepackage[colorlinks=true,
linkcolor=webgreen,
filecolor=webbrown,
citecolor=webgreen]{hyperref}

\definecolor{webgreen}{rgb}{0,.5,0}
\definecolor{webbrown}{rgb}{.6,0,0}

\usepackage{color}
\usepackage{fullpage}
\usepackage{float}

\usepackage{psfig}
\usepackage{graphics,amsmath,amssymb}
\usepackage{amsthm}
\usepackage{amsfonts}
\usepackage{latexsym}
\usepackage{epsf}

\newcommand{\Mod}[1]{\ (\mathrm{mod}\ #1)}

\begin{document}

\begin{center}
\epsfxsize=4in
\end{center}

\theoremstyle{plain}
\newtheorem{theorem}{Theorem}
\newtheorem{corollary}{Corollary}
\newtheorem{lemma}{Lemma}
\newtheorem{proposition}{Proposition}

\theoremstyle{definition}
\newtheorem{definition}{Definition}
\newtheorem{example}{Example}

\theoremstyle{remark}
\newtheorem{remark}{Remark}

\begin{center}
\vskip 1cm{\LARGE\bf 
Eigenvalues of Matrices whose Elements are Ramanujan Sums or Kloosterman Sums
}
\vskip 1cm 
\large
Noboru Ushiroya \\
National Institute of Technology, Wakayama College \\
Japan \\
ushiroya@wakayama-nct.ac.jp\\
\end{center}

\vskip .2 in

\begin{abstract} Let $c_q(n)$ be the Ramanujan sums and let $S(m,n;q)$ be the Kloosterman sums. We study the eigenvalues of $q \times q$ matrices whose $(m,n)$ entry is $c_q (m-n)$ or $S(m,n;q)$ where $q$ is a fixed positive integer.
We also study the eigenvalues of matrices whose entries are sums of Ramanujan sums or sums of Kloosterman sums.
\end{abstract}

\section{Introduction}

For $q,n \in \mathbb{N}=\{ 1,2, \cdots \}$, Ramanujan \cite{rama} defined the Ramanujan sums $c_q (n)$ by
\begin{equation*}
c_q(n)=\sum_{\substack{k=1 \\ \gcd(k,q)=1}}^q \exp \Bigl(\frac{2 \pi i kn}{q} \Bigr) , 
\end{equation*}
where $\gcd(k,q)$ is the greatest common divisor of $k$ and $q$.
For $m,n,q \in \mathbb{N}$, Kloosterman \cite{kl} defined the Kloosterman sums $S(m,n;q)$ by 
\begin{equation*}
S(m,n;q)=\sum_{\substack{k=1 \\ \gcd(k,q)=1}}^q \exp \Bigl(\frac{2 \pi i }{q} \Bigl(mk+nk^* \Bigr) \Bigr), 
\end{equation*}
where $k^*$ is the inverse of $k$ modulo $q$.
If $m = 0$ or $n = 0$, then the Kloosterman sum reduces to the Ramanujan sum. The Ramanujan sums and the Kloosterman sums have many interesting properties. However, in this paper we focus only on the eigenvalues of matrices whose elements are Ramanujan sums or Kloosterman sums.
A reason why we consider these eigenvalues is as follows.
The well known large sieve inequality (see, e.g.,  \cite{coj})
\begin{equation}
\sum_{q \leq Q} \sum_{\gcd(k,q)=1} \Bigl| \sum_{n \leq x} a_n \exp \bigl(\frac{2 \pi i kn}{q} \bigr) \Bigr|^2 \leq (Q^2+ 4 \pi x) \sum_{n \leq x} |a_n|^2 \label{eq:0-1}
\end{equation}
can be rewritten as
\begin{equation*}
 \sum_{m,n \leq x} a_m \overline{a_n} \Bigl(\sum_{q \leq Q}c_q(m-n) \Bigr) \leq  (Q^2+ 4 \pi x) \sum_{n \leq x} |a_n|^2,
\end{equation*}
where $\overline{a_n}$ is the complex conjugate of $a_n$, since
\begin{align*}
& \sum_{q \leq Q} \sum_{\gcd(k,q)=1} \Bigl|\sum_{n \leq x} a_n \exp (\frac{2 \pi i kn}{q}) \Bigr|^2 \\
=& \sum_{q \leq Q} \sum_{\gcd(k,q)=1} \sum_{m,n \leq x} a_m \overline{a_n} \exp \bigl(\frac{2 \pi i k(m-n)}{q} \bigr) \\ 
=& \sum_{m,n \leq x}  a_m \overline{a_n} \sum_{q \leq Q}  \sum_{\gcd(k,q)=1} \exp \bigl(\frac{2 \pi i k(m-n)}{q} \bigr) \\  
=& \sum_{m,n \leq x} a_m \overline{a_n} \Bigl(\sum_{q \leq Q} c_q(m-n) \Bigr).         
\end{align*}

Let $X=(X_{mn})_{m,n=1}^x$ be the matrix whose $(m,n)$ entry is $X_{mn}=\sum_{q \leq Q} c_q (m-n)$ and let $a=(a_1, a_2, \ldots , a_x)$. From (\ref{eq:0-1}) we see that the Rayleigh quotient $R(X,a):= (\sum_{m,n \leq x} a_m X_{mn} \overline{a_n})/(\sum_{n \leq x} a_n \overline{a_n})$ is not less than $Q^2+ 4 \pi x$ for any nonzero vector $a$. If we can prove that $X$ is nonnegative definite, then we have $R(X,a) \in [\lambda_m(Q,x), \lambda_M(Q,x)]$, where  $\lambda_m(Q,x),\lambda_M(Q,x) $ are the minimal and maximal positive eigenvalue of the matrix $X$ respectively, and $a$ is a vector perpendicular to the eigenspace corresponding to the eigenvalue $0$.
We cannot explicitly obtain $\lambda_M(Q,x)$ or $\lambda_m(Q,x)$ under general conditions, however, we prove in Theorem \ref{th2} that $\lambda_M(Q,x)=\lambda_m(Q,x)=x$ holds under a very restricted condition in which $x$ is the least common multiple of $1,2, \cdots, Q$. Moreover, since
\begin{align*}
\sum_{q \leq Q} \sum_{\gcd(k,q)=1} \Bigl| \sum_{n \leq x} a_n \exp \bigl(\frac{2 \pi i kn}{q} \bigr) \Bigr|^2 & \leq \sum_{q \leq Q} \sum_{\gcd(k,q)=1} (\sum_{n \leq x} |a_n|)^2 \\
& = \sum_{q \leq Q} \varphi (q) (\sum_{n \leq x} |a_n|)^2 \\
& \ll Q^2 (\sum_{n \leq x} |a_n|)^2
\end{align*}
holds for any sequence $\{a_n \}_{n =1}^x$, we have 
\begin{equation}
(\sum_{n \leq x} |a_n|)^2 \geq \frac{\lambda_m (Q,x)}{Q^2} \sum_{n \leq x} |a_n|^2 \label{eq:0-2}
\end{equation}
for any vector $(a_1, a_2, \cdots , a_x)$ perpendicular to the eigenspace corresponding to the eigenvalue $0$. If we take $a_n= \begin{cases} 1, &  \text{if $n \in S$ }; \\ 0, & \text{otherwise}, \end{cases}$ for some set $S$ and if we can show that $(a_1, a_2, \cdots , a_x)$ is perpendicular to the eigenspace corresponding to the eigenvalue $0$, then we have from (\ref{eq:0-2})
\begin{equation*}
\# \{S \cap [1,x] \} \geq \frac{\lambda_m (Q,x)}{Q^2},
\end{equation*}
from which 
\begin{equation*}
\lim_{x \to \infty} \# \{S \cap [1,x] \} = \infty
\end{equation*}
follows in the case $\lambda_m (Q,x)=x$ and $Q=x^{1/2-\varepsilon}$ for some $\varepsilon>0$.
Regretfully, we can obtain neither $\lambda_M(Q,x)$ nor $\lambda_m(Q,x)$ under general conditions, however, we expect that the above idea will lead to an estimate from below.

We also investigate the eigenvalues of $B_q=(S(m,n;q))_{m,n=1}^q$ and $Y=(Y_{mn})_{m.n=1}^x=(\sum_{q=1}^Q S (m,n;q))_{m,n=1}^x$, where $x$ is the least common multiple of $1,2, \cdots, Q$. Many results concerning the Kloosterman sums are obtained by many mathematicians hitherto. For example, it is well known that Weil's bound \cite{we}
\begin{equation*}
|S(a,b;q)| \leq \tau(q) \gcd (a,b,q) \sqrt{q}
\end{equation*}
holds where $\tau (q)=\sum_{d | q} 1$. Kuznetsov's bound \cite{ku}
\begin{equation*}
\sum_{q \leq Q} \frac{S(a,b;q)}{q} \ll Q^{\frac{1}{6} + \varepsilon} 
\end{equation*}
is also well known. However, as for the eigenvalues of matrices whose elements are the Kloosterman sums, to my knowledge, few results are known. We obtain in Theorem \ref{th3} and \ref{th4} the eigenvalues of matrices whose elements are Kloosterman sums. Of course there is no relation between the eigenvalues and Weil's bound or Kuznetsov's bound, however, we expect that our study contributes somewhat to the theory of Kloosterman sums.

\section{Properties of Ramanujan sums and  Kloosterman sums}

In this section we show some properties of the Ramanujan sums $c_q(n)$ and the Kloosterman sums $S(m,n;q)$.
Let $\delta (m,n)= \begin{cases} 1, &  \text{if $m=n$ }; \\ 0, & \text{if $m \neq n$}, \end{cases}$ and let
 $\varphi(q)=\# \{k \leq q ; \gcd(k,q)=1 \}$ be the Euler totient function. 
The following properties of the Ramanujan sums are well known (see, e.g., \cite{co,Mc,ramM,sc}).
\begin{align}
& c_q(n)=c_q(-n),  \label{eq-ram-01} \\
& c_q(0)=\varphi(q), \label{eq-ram-02} \\
& c_q(n)=c_q(n') \quad \text{if} \quad n \equiv n'  \Mod{q}, \label{eq-ram-03} \\
& \sum_{k=1}^q c_q (k)=0.  \label{eq-ram-1}
\end{align}

The following properties of the Kloosterman sums are also well known (see, e.g., \cite{iw}).
\begin{align*}
& S(m,n;q)=S(n,m;q), \\
& S(m,n;q)=S(m',n';q) \quad \text{if} \quad m \equiv m' \Mod{q} \quad \text{and} \quad n \equiv n' \Mod{q}.
\end{align*}

We write $d|n$ if $d$ divides $n$. We first prove the following lemma.
\begin{lemma}
\label{lem1}
If $x \in \mathbb{N}, \ q \mid x$ and $r \mid x$, then we have 
\begin{align}
& \sum_{a=1}^x c_q(m-a) c_r(a-n) =\begin{cases} x c_q(m-n), & \text{if $q=r$}; \\ 0, & \text{otherwise}, \end{cases} \label{eq:lem1-1}  \\
& \sum_{a=1}^x S(m,a;q) S(a,n;r) =\begin{cases} x c_q(m-n), & \text{if $q=r$}; \\ 0, & \text{otherwise}. \end{cases} \label{eq:lem1-2} 
\end{align}
\end{lemma}

In order to prove Lemma \ref{lem1}, we prepare the following lemma.
\begin{lemma}
\label{lem2}
If $x \in \mathbb{N}, \ q \mid x$, $r \mid x$ and $\gcd(k,q)=\gcd(\ell,r)=1$, then we have
\begin{align*}
\sum_{a=1}^x \exp \Bigl(2 \pi i \Bigl(\frac{k}{q}+\frac{\ell}{r} \Bigr)a \Bigr)=\begin{cases} x, & \text{if $q=r$ and $k+\ell \equiv 0 \Mod q$}; \\ 0, & \text{otherwise}. \end{cases} 
\end{align*}
\end{lemma}

\begin{proof}
The proof proceeds along the same line as the proof of  of Theorem 1 in \cite{ramM}.
We first note that the following holds by the partial-sum formula of a geometric series.
\begin{equation*}
\sum_{a=1}^x \exp \Bigl(2 \pi i \Bigl(\frac{k}{q}+\frac{\ell}{r} \Bigr)a \Bigr)=\begin{cases} x, & \text{if $k/q+\ell/r$ is an integer}; \\ 0, & \text{otherwise}. \end{cases}
\end{equation*}

If $q \neq r$, then $k/q+\ell/r=(kr+\ell q)/qr $ is not an integer since $\gcd(k,q)=\gcd(\ell,r)=1$. If $q=r$, then $k/q+\ell/r=(k+\ell)/q $ is an integer or not according to whether $k+\ell$ is congruent to $0 \Mod{q}$ or not. This completes the proof of Lemma \ref{lem2}.
\end{proof}

Now we can prove Lemma \ref{lem1}.

\begin{proof}[Proof of Lemma \ref{lem1}]
By the definition of the Ramanujan sums we have
\begin{align*}
\sum_{a=1}^x c_q(m-a) c_r(a-n) =& \sum_{a=1}^x \sum_{\substack{k=1 \\ \gcd(k,q)=1}}^q \exp \Bigl(\frac{2 \pi i k(m-a)}{q} \Bigr) \sum_{\substack{\ell=1 \\ \gcd(\ell,r)=1}}^r \exp \Bigl(\frac{2 \pi i \ell (a-n)}{r} \Bigr) \\
=& \sum_{\substack{k=1 \\ \gcd(k,q)=1}}^q \sum_{\substack{\ell=1 \\ \gcd(\ell,r)=1}}^r \exp \Bigl(2 \pi i (\frac{km}{q}-\frac{\ell n}{r}) \Bigr) \sum_{a=1}^x \exp \Bigl(2 \pi i (-\frac{k}{q}+\frac{\ell}{r})a \Bigr).
\end{align*}

Since
\begin{equation*}
\sum_{a=1}^x \exp \Bigl(2 \pi i (-\frac{k}{q}+\frac{\ell}{r})a \Bigr)=\begin{cases} x, & \text{if $q=r$ and $-k+\ell \equiv 0 \Mod{q}$}; \\ 0, & \text{otherwise} \end{cases} 
\end{equation*}
holds by Lemma \ref{lem2}, we have
\begin{align*}
\sum_{a=1}^x c_q(m-a) c_r(a-n) =& x \delta(q,r) \sum_{\substack{k=1 \\ \gcd(k,q)=1}}^q \sum_{\substack{\ell=1 \\ \gcd(\ell,q)=1 \\ \ell \equiv k \Mod q}}^q \exp \Bigl(2 \pi i (\frac{km}{q}-\frac{\ell n}{q}) \Bigr) \\
=& x \delta(q,r)\sum_{\substack{k=1 \\ \gcd(k,q)=1}}^q \exp \Bigl(\frac{2 \pi i}{q} (m-n) k \Bigr) \\
=& x \delta(q,r) c_q(m-n)=\begin{cases} x c_q(m-n), & \text{if $q=r$}; \\ 0, & \text{otherwise}. \end{cases}
\end{align*}
Therefore (\ref{eq:lem1-1}) holds.

Next we prove (\ref{eq:lem1-2}). We have
\begin{align*}
\sum_{a=1}^x S(m,a;q) S(a,n;r) =& \sum_{a=1}^x \sum_{\substack{k=1 \\ \gcd(k,q)=1}}^q \exp \Bigl(\frac{2 \pi i }{q} (m k+a k^*) \Bigr) \sum_{\substack{\ell=1 \\ \gcd(\ell,r)=1}}^r \exp \Bigl(\frac{2 \pi i }{r} (a \ell+n \ell^*) \Bigr) \\
= & \sum_{\substack{k=1 \\ \gcd(k,q)=1}}^q \sum_{\substack{\ell=1 \\ \gcd(\ell,r)=1}}^r \exp \Bigl(2 \pi i (\frac{mk}{q}+\frac{n \ell^* }{r}) \Bigr) \sum_{a=1}^x \exp \Bigl(2 \pi i (\frac{k^*}{q}+\frac{\ell}{r})a \Bigr).
\end{align*}

Since
\begin{equation*}
\sum_{a=1}^x \exp \Bigl(2 \pi i (\frac{k^*}{q}+\frac{\ell}{r})a \Bigr)=\begin{cases} x, & \text{if $q=r$ and $k^*+\ell \equiv 0 \Mod q$}; \\ 0, & \text{otherwise} \end{cases} 
\end{equation*}
holds by Lemma \ref{lem2} and since $k^*+\ell \equiv 0 \Mod{q}$ is equivalent to $\ell^* \equiv -k \Mod{q}$, we have
\begin{align*}
\sum_{a=1}^x S(m,a;q) S(a,n;r) =& x \delta(q,r) \sum_{\substack{k=1 \\ \gcd(k,q)=1}}^q \sum_{\substack{\ell=1 \\ \gcd(\ell,q)=1 \\ \ell^* \equiv -k \Mod q}}^q \exp \Bigl(2 \pi i (\frac{mk}{q}+\frac{n \ell^*}{q}) \Bigr) \\
=& x \delta(q,r)\sum_{\substack{k=1 \\ \gcd(k,q)=1}}^q \exp \Bigl(\frac{2 \pi i}{q} (m-n) k \Bigr) \\
=& x \delta(q,r) c_q(m-n)=\begin{cases} x c_q(m-n), & \text{if $q=r$}; \\ 0, & \text{otherwise}. \end{cases} 
\end{align*}
Therefore (\ref{eq:lem1-2}) holds. This completes the proof of Lemma \ref{lem1}.
\end{proof}

Consider the function $\widetilde{\varphi}$ defined by
\begin{equation*}
\widetilde{\varphi} (q)=\# \{ 1 \leq k \leq q; \ \gcd(k,q)=1 , \ \ k+k^* \equiv 0 \Mod{q} \} , 
\end{equation*}
where $k^*$ is the inverse of $k \Mod{q}$. 
It is easy to see that $\widetilde{\varphi} (q)$ can be rewritten as 
\begin{equation*}
\widetilde{\varphi} (q)=\# \{ 1 \leq k \leq q; \ \gcd(k,q)=1 , \ \ k^2 \equiv -1 \Mod{q} \}.
\end{equation*}

Recall that an arithmetic function $f:\mathbb{N} \mapsto \mathbb{C}$ is said to be a multiplicative function if $f$ satisfies
\begin{equation}
\nonumber
f(m n)=f(m)f(n) 
\end{equation}
for any $m,n \in  \ \mathbb{N} $ satisfying $\gcd (m,n)=1$.
We let $\mathcal{P}$ denote the set of prime numbers. We have the following lemma.

\begin{lemma}
\label{lem3}
The function $q \mapsto \widetilde{\varphi} (q)$ is a multiplicative function satisfying
\begin{align*}
&\widetilde{\varphi} (2)  = 1, \\
&\widetilde{\varphi} (2^e)  = 0 \ \ \text{if \ $e \geq 2$},  \\
&\widetilde{\varphi} (p^e)  = 0 \ \ \text{if \ $p \in \mathcal{P}, \ \ p \equiv -1 \Mod{4}$ \ and \ $e \geq 1 $},  \\
&\widetilde{\varphi} (p^e)  = 2 \ \ \text{if \ $p \in \mathcal{P},  \ \ p \equiv 1 \Mod{4}$ \ and \ $e \geq 1 $}.
\end{align*}
\end{lemma}

\begin{proof}
Multiplicativity follows from the Chinese remainder theorem. The other parts follow from \cite[p.\ 91]{vi}.
\end{proof}

Next we prove the following lemma.
\begin{lemma}
\label{lem4}
If $x \in \mathbb{N}, \ q \mid x$ and $r \mid x$, then we have 
\begin{equation*}
\sum_{m=1}^x \sum_{a=1}^x c_q(m-a) S(a,m;r) =\begin{cases} x^2 \widetilde{\varphi} (q) , & \text{if $q=r$}; \\ 0, & \text{otherwise}. \end{cases}   
\end{equation*}
\end{lemma}

\begin{proof}
First, we have
\begin{align}
& \sum_{m=1}^x \sum_{a=1}^x c_q(m-a) S(a,m;r) \nonumber \\
&= \sum_{m=1}^x \sum_{a=1}^x \sum_{\substack{k=1 \\ \gcd(k,q)=1}}^q \exp \Bigl(\frac{2 \pi i k(m-a)}{q} \Bigr) \sum_{\substack{\ell=1 \\ \gcd(\ell,r)=1}}^r \exp \Bigl(\frac{2 \pi i }{r} (a \ell+m \ell^*) \Bigr) \nonumber \\
&= \sum_{\substack{k=1 \\ \gcd(k,q)=1}}^q \sum_{\substack{\ell=1 \\ \gcd(\ell,r)=1}}^r \sum_{m=1}^x \exp \Bigl(2 \pi i (\frac{k}{q}+\frac{\ell^* }{r})m \Bigr) \sum_{a=1}^x \exp \Bigl(2 \pi i (-\frac{k}{q}+\frac{\ell}{r})a \Bigr). \label{eq:lem4-1}
\end{align}

Since
\begin{align*}
&\sum_{m=1}^x \exp \Bigl(2 \pi i (\frac{k}{q}+\frac{\ell^*}{r})m \Bigr)=\begin{cases} x, & \text{if $q=r$ and $k+\ell^* \equiv 0 \Mod q$}; \\ 0, & \text{otherwise}, \end{cases} \\
&\sum_{a=1}^x \exp \Bigl(2 \pi i (-\frac{k}{q}+\frac{\ell}{r})a \Bigr)=\begin{cases} x, & \text{if $q=r$ and $-k+\ell \equiv 0 \Mod q$}; \\ 0, & \text{otherwise} \end{cases} 
\end{align*}
hold by Lemma \ref{lem2}, and since
\begin{align*}
& k+\ell^* \equiv 0 \Mod{q} \quad \text{and} \quad -k+\ell \equiv 0 \Mod{q}
\intertext{is equivalent to}
& \ell \equiv k \Mod{q} \quad \text{and} \quad k+k^* \equiv 0 \Mod{q}, 
\end{align*}
we see that (\ref{eq:lem4-1}) is equal to
\begin{align*}
&  x^2 \delta(q,r) \sum_{\substack{k=1 \\ \gcd(k,q)=1 \\ k+k^* \equiv 0 \Mod q}}^q \sum_{\substack{\ell=1 \\ \gcd(\ell,q)=1 \\ \ell \equiv k \Mod q}}^q 1 \quad
= \quad x^2 \delta(q,r)\sum_{\substack{k=1 \\ \gcd(k,q)=1 \\ k+k^* \equiv 0 \Mod q }}^q 1 \\
&= x^2 \delta(q,r) \widetilde{\varphi} (q) = \begin{cases} x^2 \widetilde{\varphi} (q) , & \text{if $q=r$}; \\ 0, & \text{otherwise}. \end{cases} 
\end{align*}
This completes the proof of Lemma \ref{lem4}.
\end{proof}

\section{Some results}

\subsection{The case of Ramanujan sums}
Let $c_q(n)$ be the Ramanujan sums.
We first consider the following $q \times q$ matrix 
\begin{equation*}
A_q=(c_q (m-n))_{m,n=1}^q=\left( \begin{array}{ccccc} c_q(0) & c_q(1) & c_q(2) & \cdots & c_q(q-1)    \\ c_q(1) & c_q(0) & c_q(1) & \cdots & c_q(q-2) \\ c_q(2) & c_q(1) & c_q(0) & \cdots & c_q(q-3) \\ \vdots & \vdots & \vdots & \ddots & \vdots   \\ c_q(q-1) & c_q(q-2) & c_q(q-3) & \cdots & c_q(0)  \\ \end{array}  \right),
\end{equation*}
where $q$ is a fixed positive integer.
We begin with the following lemma.

\begin{lemma} 
\label{lem5}
Let $A_q=(c_q (m-n))_{m,n=1}^q$. For every integer $j \geq 2$, we have 
\begin{equation}
 A_q^j=q^{j-1}A_q .  \label{eq:lem5-1}
\end{equation}
\end{lemma}

\begin{proof} 
We first prove (\ref{eq:lem5-1}) for $j=2$. From Lemma \ref{lem1} we see that the $(m,n)$ entry of $A_q^2$ equals
\begin{equation*}
\sum_{a=1}^q c_q(m-a) c_q(a-n)= q c_q (m-n),
\end{equation*}
which is equal to the $(m,n)$ entry of $q A_q$.
The general case $j \geq 2$ follows by induction.
\end{proof}

We let $\mathrm{tr} (M)=\sum_{m=1}^k M_{mm}$ denote the trace of $M$ where $M$ is a $k \times k$ matrix.
Next we consider the trace of $A_q^j$ where $j \in \mathbb{N}$. 

\begin{lemma} 
\label{lem6}
Let $A_q=(c_q (m-n))_{m,n=1}^q$. Then we have for every $j \in \mathbb{N}$
\begin{equation}
\mathrm{tr} (A_q^j)  = q^j \varphi(q). \label{eq:lem6-1}
\end{equation}
\end{lemma}

\begin{proof}
Since $c_q(0)=\varphi(q)$, we have
\begin{equation*}
\mathrm{tr} (A_q) =\sum_{m=1}^q c_q(m-m)=\sum_{m=1}^q c_q (0) = \sum_{m=1}^q \varphi(q)= q \varphi(q).
\end{equation*}

Therefore (\ref{eq:lem6-1}) holds for $j=1$. If $j \geq 2$, then we have by Lemma \ref{lem5} and the above result
\begin{align*}
\mathrm{tr} (A_q^j) =q^{j-1} \mathrm{tr} (A_q)= q^{j-1} q \varphi(q)=q^j \varphi(q).
\end{align*}
\end{proof}

Let $E_q$ denote the $q \times q$ identity matrix and let $\det (M)$ denote the determinant of $M$ where $M$ is a square matrix.
We prove the following theorem.

\begin{theorem} 
\label{th1} 
Let $A_q=(c_q (m-n))_{m,n=1}^q$. Then the characteristic polynomial of $A_q$ is
\begin{equation*}
\det(\lambda E_q-A_q)=\lambda^{q-\varphi(q)} (\lambda -q)^{\varphi(q)}. 
\end{equation*}

Especially, the matrix $A_q$ has eigenvalues $0,q $ with multiplicity $q-\varphi(q), \varphi(q)$, respectively.
\end{theorem}

\begin{proof}
We first suppose $\lambda > q$. Since $\det (\exp (M))=\exp( \mathrm{tr} (M))$ holds for any square matrix $M$, we have
\begin{equation*}
\det(\lambda E_q-A_q)= \lambda^q \det( E_q-\frac{1}{\lambda} A_q)= \lambda^q \exp( \mathrm{tr} (\log (E_q-\frac{1}{\lambda} A_q))).
\end{equation*}

Since
\begin{equation*}
\log (E_q-\frac{1}{\lambda} A_q )= \sum_{j=1}^\infty \frac{(-1)^{j-1}}{j} (-\frac{1}{\lambda} A_q)^j= \sum_{j=1}^\infty \frac{(-1)^{j-1}}{j} (-\frac{1}{\lambda})^j A_q^j
\end{equation*}
holds, we have by Lemma \ref{lem6}
\begin{align*}
\mathrm{tr} (\log (E_q-\frac{1}{\lambda} A_q ))=& \sum_{j=1}^\infty \frac{(-1)^{j-1}}{j} (-\frac{1}{\lambda})^j \mathrm{tr} (A_q^j) \\
=& \sum_{j=1}^\infty \frac{(-1)^{j-1}}{j} (-\frac{1}{\lambda})^j q^j \varphi(q) = \varphi(q) \log(1-\frac{q}{\lambda}).
\end{align*}

Therefore we obtain
\begin{align}
\det(\lambda E_q-A_q)=& \lambda^q \exp \Bigl(\varphi(q) \log(1-\frac{q}{\lambda}) \Bigr) \nonumber \\
 =&  \lambda^q (1-\frac{q}{\lambda})^{\varphi(q)} = \lambda^{q-\varphi(q)} (\lambda -q)^{\varphi(q)}. \label{eq:th1-1}
\end{align}

Since (\ref{eq:th1-1}) holds for any $\lambda > q$ and both sides of (\ref{eq:th1-1}) are polynomials of $\lambda$, we see that (\ref{eq:th1-1}) holds for any $\lambda$.
This completes the proof of Theorem \ref{th1}.
\end{proof}

\begin{remark}
A $q \times q$ matrix $C$ is called a circulant matrix if $C$ takes the form
\begin{equation}
C=
\begin{pmatrix}
c(0)     & c(q-1) & \cdots  & c(2) & c(1)  \\
c(1) & c(0)    & c(q-1) &         & c(2)  \\
\vdots  & c(1)& c(0)    & \ddots  & \vdots   \\
c(q-2)  &        & \ddots & \ddots  & c(q-1)   \\
c(q-1)  & c(q-2) & \cdots  & c(1) & c(0) \\
\end{pmatrix}. \label{eq:rem1-1}
\end{equation}

We immediately see that $A_q=(c_q (m-n))_{m,n=1}^q$ is a circulant matrix since $c_q(a)=c_q(-a)=c_q(q-a)$ holds for any $1 \leq a \leq q$. It is well known that the eigenvalues of (\ref{eq:rem1-1}) are given by
\begin{equation}
\lambda_j=\sum_{a=1}^q c(q-a) \omega_j^a, \label{eq:circu}
\end{equation}
where $\omega_j=\exp (2 \pi i j / q )$ and $1 \leq j \leq q$. Substituting $c(q-a)$ by $c_q(q-a)=c_q(a)$, we see that the eigenvalues of $A_q$ are equal to
\begin{align*}
\lambda_j&= \sum_{a=1}^{q} c_q(a) \omega_j^a=\sum_{a=1}^{q} \sum_{\substack{k=1 \\ \gcd(k,q)=1}}^q \exp \Bigl(\frac{2 \pi i ka}{q} \Bigr) \exp \Bigl(\frac{2 \pi i ja}{q} \Bigr) \\
&= \sum_{\substack{k=1 \\ \gcd(k,q)=1}}^q \sum_{a=1}^{q} \exp \Bigl(\frac{2 \pi i a(k+j)}{q} \Bigr)=\begin{cases} q , & \text{if $\gcd(j,q)=1$}; \\ 0, & \text{otherwise}. \end{cases}
\end{align*}
Thus we can easily obtain the eigenvalues of $A_q$. However, we proved Theorem 1 without using a property of a circulant matrix since the proof of Theorem 1 applies also to that of Theorem 3. 

As for the eigenvectors of a circulant matrix, it is also well known that the eigenvectors of (\ref{eq:rem1-1}) are given by
\begin{equation}
v_j=   (1,\omega_j,\omega_j^2, \cdots, \omega_j^{q-1}), \label{eq:circu2}
\end{equation}
where $\omega_j=\exp (2 \pi i j / q )$ and $1 \leq j \leq q$. 
\end{remark}

From the above remark we immediately obtain the following corollary.

\begin{corollary}
Let $A_q=(c_q (m-n))_{m,n=1}^q$. Then the eigenvalues $\lambda_j$ and the eigenvectors $v_j \  (1 \leq j \leq q)$  of $A_q$ are given by the following.

If $\gcd(j,q)=1$, then
\begin{align*}
& \lambda_j=q, \\
& v_j= (\underbrace{1,\omega_j,\omega_j^2, \cdots, \omega_j^{q-1}}_{q}),
\intertext{and if $\gcd(j,q)>1$, then}
& \lambda_j=0, \\
& v_j= (\underbrace{1,\omega_j,\omega_j^2, \cdots, \omega_j^{q'-1}}_{q'},\underbrace{1,\omega_j,\omega_j^2, \cdots, \omega_j^{q'-1}}_{q'}, \cdots,\underbrace{1,\omega_j,\omega_j^2, \cdots, \omega_j^{q'-1}}_{q'}),
\end{align*}
where  $\omega_j=\exp (2 \pi i j / q )$ and $q'=q/\gcd(j,q)$.
\end{corollary}

Next we consider the $ x \times x$ matrix $X=(X_{mn})_{m,n=1}^x=(\sum_{q=1}^Q c_q (m-n))_{m,n=1}^x$ where $Q$ is a fixed positive integer and $x$ is the least common multiple of $1,2, \cdots, Q$.
We proceed along similar lines to the case of $A_q$.

\begin{lemma} 
\label{lem7}
Let $X=(X_{mn})_{m,n=1}^x=(\sum_{q=1}^Q c_q (m-n))_{m,n=1}^x$ where $x$ is the least common multiple of $1,2, \cdots, Q$. For every integer $j \geq 2$, we have 
\begin{equation}
X^j=x^{j-1}X . \label{eq:lem7-1}
\end{equation}
\end{lemma}

\begin{proof}
First, we prove (\ref{eq:lem7-1}) for $j=2$. From Lemma \ref{lem1}, we see that the $(m,n)$ component of $X^2$ equals
\begin{align*}
(X^2)_{mn}= & \sum_{k=1}^x X_{m k} X_{k n} =\sum_{k=1}^x (\sum_{q=1}^Q  c_q (m-k)) (\sum_{r=1}^Q  c_r (k-n)) \\
 =& \sum_{q=1}^Q \sum_{r=1}^Q \sum_{k=1}^x c_q(m-k) c_r(k-n) = \sum_{q=1}^Q \sum_{r=1}^Q x \delta (q,r) c_q(m-n) \\
= & x \sum_{q=1}^Q  c_q(m-n)=x X_{mn}.
\end{align*}

Hence, (\ref{eq:lem7-1}) holds for $j=2$.
The general case $j \geq 2$ follows by induction.
\end{proof}

Next we consider the trace of $X^j$ where $j \in \mathbb{N}$. 
We set $\Phi(Q)=\sum_{q=1}^Q \varphi(q)$.

\begin{lemma} 
\label{lem8}
Let $X=(X_{mn})_{m,n=1}^x=(\sum_{q=1}^Q c_q (m-n))_{m,n=1}^x$ where $x$ is the least common multiple of $1,2, \cdots, Q$. Then  we have for every $j \in \mathbb{N}$
\begin{equation}
\mathrm{tr} (X^j)  = x^j \Phi(Q). \label{eq:lem8-1}
\end{equation}
\end{lemma}

\begin{proof}
Since $c_q(0)=\varphi(q)$, we have 
\begin{align*}
\mathrm{tr} (X) =\sum_{m=1}^x X_{mm}=\sum_{m=1}^x \sum_{q=1}^Q  c_q (0) =\sum_{m=1}^x \sum_{q=1}^Q  \varphi(q) =x \Phi(Q).
\end{align*}

Thus (\ref{eq:lem8-1}) holds for $j=1$. If $j \geq 2$, then we have by Lemma \ref{lem7} and the above result
\begin{align*}
\mathrm{tr} (X^j) =x^{j-1} \mathrm{tr} (X)= x^{j-1} x \Phi(Q)=x^j \Phi(Q).
\end{align*}

Therefore (\ref{eq:lem8-1}) holds for  every $j \in \mathbb{N}$.
\end{proof}

Let $E_x$ be the $x \times x$ identity matrix where $x$ is a positive integer.
We prove the following theorem.

\begin{theorem} 
\label{th2} 
Let $X=(X_{mn})_{m,n=1}^x=(\sum_{q=1}^Q c_q (m-n))_{m,n=1}^x$ where $x$ is the least common multiple of $1,2, \cdots, Q$. Then the characteristic polynomial of $X$ is
\begin{equation*}
\det(\lambda E_x-X)=\lambda^{x-\Phi(Q)} (\lambda -x)^{\Phi(Q)}.
\end{equation*}

Especially, the matrix $X$ has eigenvalues $0,x $ with multiplicity  $x-\Phi(Q), \Phi(Q)$,  respectively.
\end{theorem}

\begin{proof}
The proof proceeds along the same lines as the proof of Theorem \ref{th1}.
We first suppose $\lambda > x$. Since $\det (\exp (M))=\exp( \mathrm{tr} (M))$ holds for any square matrix $M$, we have
\begin{equation*}
\det(\lambda E_x-X)= \lambda^x \det( E_x-\frac{1}{\lambda} X)= \lambda^x \exp( \mathrm{tr} (\log (E_x-\frac{1}{\lambda} X))) .
\end{equation*}

Since
\begin{equation*}
\log (E_x-\frac{1}{\lambda} X )= \sum_{j=1}^\infty \frac{(-1)^{j-1}}{j} (-\frac{1}{\lambda} X)^j= \sum_{j=1}^\infty \frac{(-1)^{j-1}}{j} (-\frac{1}{\lambda})^j X^j
\end{equation*}
holds, we have by Lemma \ref{lem8}
\begin{align*}
\mathrm{tr} (\log (E_x-\frac{1}{\lambda} X ))=& \sum_{j=1}^\infty \frac{(-1)^{j-1}}{j} (-\frac{1}{\lambda})^j \mathrm{tr} (X^j) \\
=& \sum_{j=1}^\infty \frac{(-1)^{j-1}}{j} (-\frac{1}{\lambda})^j x^j \Phi(Q) = \Phi(Q) \log(1-\frac{x}{\lambda}).
\end{align*}

Therefore we have
\begin{align}
\det(\lambda E_x-X)=& \lambda^x \exp \Bigl(\Phi(Q) \log(1-\frac{x}{\lambda}) \Bigr)  \nonumber \\
 =&  \lambda^x (1-\frac{x}{\lambda})^{\Phi(Q)} = \lambda^{x-\Phi(Q)} (\lambda -x)^{\Phi(Q)}. \label{eq:th2-1}
\end{align}

Since (\ref{eq:th2-1}) holds for any $\lambda > x$ and both sides of (\ref{eq:th2-1}) are polynomials of $\lambda$, (\ref{eq:th2-1}) holds for any $\lambda$.
This completes the proof of Theorem \ref{th2}.
\end{proof}

\begin{remark}
$X$ as well as $A_q$ is a circulant matrix. By (\ref{eq:circu}) we see that the eigenvalues $\lambda_j \ \ (1 \leq j \leq x)$ of $X$ are given by
\begin{align*}
\lambda_j=\sum_{a=1}^{x} \sum_{q=1}^{Q} c_q(a) \omega_j^a=& \sum_{q=1}^{Q} \sum_{a=1}^{x} \sum_{\substack{k=1 \\ \gcd(k,q)=1}}^q \exp \Bigl(2 \pi i a \bigl(\frac{k}{q}+\frac{j}{x} \bigr) \Bigr)\\
=& \sum_{q=1}^{Q}  \sum_{\substack{k=1 \\ \gcd(k,q)=1}}^q \sum_{a=1}^{x} \exp \Bigl(2 \pi i a \bigl(\frac{k}{q}+\frac{j'}{x'} \bigr) \Bigr),
\end{align*}
where $\omega_j=\exp(2 \pi i j/x)$, $j'=j/\gcd(j,x)$, and $x'=x/\gcd(j,x)$. Since it follows from Lemma \ref{lem2} that
\begin{equation*}
\sum_{a=1}^{x} \exp \Bigl(2 \pi i a \bigl(\frac{k}{q}+\frac{j'}{x'} \bigr) \Bigr)= \begin{cases} x , & \text{if $q=x'$ and $k+j' \equiv 0 \Mod{q}$}; \\ 0, & \text{otherwise}, \end{cases}
\end{equation*}
we have
\begin{equation}
\lambda_j=\sum_{q=1}^{Q}  \sum_{\substack{k=1 \\ \gcd(k,q)=1 \\ k+j' \equiv 0 \Mod{q}}}^q  x \delta (q,x') = x \sum_{q=1}^{Q} \delta (q,x') \sum_{\substack{k=1 \\ \gcd(k,x')=1 \\ k \equiv -j' \Mod{q}}}^q 1   = \begin{cases} x , & \text{if $x' \leq Q$}; \\ 0, & \text{otherwise}. \end{cases} \label{eq:rem2-1}
\end{equation}
Therefore Theorem \ref{th2} can be recovered from the following lemma.
\end{remark}

\begin{lemma}
\label{lem9}
Let $Q \in \mathbb{N}$ and $x$ is the least common multiple of $1,2, \cdots, Q$. Then we have
\begin{equation*}
\# \{1 \leq j \leq x ; \ \frac{x}{\gcd(j,x)} \leq Q \}=\Phi(q).
\end{equation*}
\end{lemma}
\begin{proof}
It is easy to see that
\begin{align*}
& \# \{1 \leq j \leq x ; \ \frac{x}{\gcd(j,x)} \leq Q \} \\
=& \sum_{q=1}^Q \# \{1 \leq j \leq x ; \ \frac{x}{\gcd(j,x)} =q \} \\
=& \sum_{q=1}^Q \# \{1 \leq j \leq x ; \ \gcd(j,x)=\frac{x}{q} \} \\
=& \sum_{q=1}^Q \# \{1 \leq j \leq x ; \ j=\frac{x}{q} k \text{ for some }   1 \leq k \leq q \text{ such that } \gcd(k,q)=1 \} \\
=&\sum_{q=1}^Q \# \{1 \leq k \leq q ; \  \gcd(k,q)=1 \}= \sum_{q=1}^Q \varphi(q)=\Phi(Q).
\end{align*}
\end{proof}

Now we have the following corollary by (\ref{eq:circu2}) and (\ref{eq:rem2-1}).
\begin{corollary}
Let $X=(X_{mn})_{m,n=1}^x=(\sum_{q=1}^Q c_q (m-n))_{m,n=1}^x$ where $x$ is the least common multiple of $1,2, \cdots, Q$. Then the eigenvalues $\lambda_j$ and the eigenvectors $v_j \ (1 \leq j \leq x)$ of $X$ are given by the following.

If $x'=x / \gcd(j,x) \leq Q$, then
\begin{align*}
& \lambda_j=x, \\
& v_j=(\underbrace{1,\omega_j,\omega_j^2, \cdots, \omega_j^{x'-1}}_{x'},\underbrace{1,\omega_j,\omega_j^2, \cdots, \omega_j^{x'-1}}_{x'}, \cdots,\underbrace{1,\omega_j,\omega_j^2, \cdots, \omega_j^{x'-1}}_{x'}),
\intertext{and if $x'=x / \gcd(j,x) > Q$, then}
& \lambda_j=0, \\
& v_j= (\underbrace{1,\omega_j,\omega_j^2, \cdots, \omega_j^{x'-1}}_{x'},\underbrace{1,\omega_j,\omega_j^2, \cdots, \omega_j^{x'-1}}_{x'}, \cdots,\underbrace{1,\omega_j,\omega_j^2, \cdots, \omega_j^{x'-1}}_{x'}),
\end{align*}
where  $\omega_j=\exp (2 \pi i j / x )$.
\end{corollary}

\subsection{The case of Kloosterman sums}

Let $S(m,n;q)$ be the Kloosterman sums.
We first consider the following $q \times q$ matrix 
\begin{equation*}
B_q=(S(m,n;q))_{m,n=1}^q,
\end{equation*}
where $q$ is a fixed positive integer.
We begin with the following lemma.

\begin{lemma}
\label{lem10}
Let $B_q=(S(m,n;q))_{m,n=1}^q$. Then we have
\begin{equation*}
B_q^2=qA_q.
\end{equation*}
\end{lemma}

\begin{proof}
From Lemma \ref{lem1} we see that the $(m,n)$ entry of $B_q^2$ equals
\begin{equation*}
\sum_{a=1}^q S(m,a;q) S(a,n;q)= q c_q (m-n),
\end{equation*}
which is equal to the $(m,n)$ entry of $q A_q$.
\end{proof}

Next we consider the trace of $B_q^j$ where $j \in \mathbb{N}$.

\begin{lemma}
\label{lem11}
Let $B_q=(S(m,n;q))_{m,n=1}^q$. For $ j \in \mathbb{N}$, we have
\begin{align}
 & \mathrm{tr} (B_q^{2j})  =q^{2j} \varphi (q), \label{eq:lem11-1} \\
 & \mathrm{tr} (B_q^{2j-1})  =q^{2j-1} \widetilde{\varphi} (q). \label{eq:lem11-2}
\end{align}
\end{lemma}

\begin{proof}
We first prove $\mathrm{tr} (B_q)  =q \widetilde{\varphi} (q)$. It follows that
\begin{align}
\mathrm{tr} (B_q) = & \sum_{m=1}^q S(m,m;q)= \sum_{m=1}^q \sum_{\substack{k=1 \\ \gcd(k,q)=1}}^q \exp (\frac{2 \pi i }{q} (mk+mk^*))  \nonumber \\
= & \sum_{\substack{k=1 \\ \gcd(k,q)=1}}^q \sum_{m=1}^q  \exp (\frac{2 \pi i }{q} (k+k^*)m) . \label{eq:lem11-3}
\end{align}

Since it follows that
\begin{equation*}
\sum_{m=1}^q  \exp (\frac{2 \pi i }{q} (k+k^*)m) =\begin{cases} q, & \text{if $k+k^* \equiv 0 \Mod q $}; \\ 0, & \text{otherwise}, \end{cases}
\end{equation*}
we see that $(\ref{eq:lem11-3})$ is equal to
\begin{align*}
 &q \sum_{\substack{k=1 \\ \gcd(k,q)=1 \\ k+k^* \equiv 0 \Mod q}}^q 1 \\
=& q \ \# \{ 1 \leq k \leq q; \gcd(k,q)=1 , \ \ k+k^* \equiv 0 \Mod q \}=q \widetilde{\varphi } (q).
\end{align*}

Therefore $\mathrm{tr} (B_q)  =q \widetilde{\varphi} (q)$ holds. \\

Next we prove (\ref{eq:lem11-1}). By Lemma \ref{lem6} and Lemma \ref{lem10} we have
\begin{equation*}
\mathrm{tr} (B_q^{2j}) =\mathrm{tr} ((B_q^{2})^j) = \mathrm{tr} ((qA_q)^j)= q^j \mathrm{tr} ( A_q^j)= q^j q^j \varphi (q)=q^{2j} \varphi(q).
\end{equation*}

Therefore (\ref{eq:lem11-1}) holds. 

Next we prove (\ref{eq:lem11-2}) for $j \geq 2$.
We have by Lemma \ref{lem5} and Lemma \ref{lem10}
\begin{align}
\mathrm{tr} (B_q^{2j-1}) & = \mathrm{tr} (B_q^{2(j-1)} B_q)=\mathrm{tr} ((q A_q)^{j-1} B_q) \nonumber \\
&= \mathrm{tr} (q^{j-1} q^{j-2}A_q  B_q)= q^{2j-3} \mathrm{tr} (A_q  B_q). \label{eq:lem11-4}
\end{align}

Since $\mathrm{tr} (A_q  B_q)=q^2 \widetilde{\varphi} (q)$ holds by Lemma \ref{lem4}, we see that (\ref{eq:lem11-4}) is equal to
\begin{equation*}
q^{2j-3} q^2 \widetilde{\varphi} (q)=q^{2j-1}  \widetilde{\varphi} (q).
\end{equation*}
This completes the proof of Lemma \ref{lem11}.
\end{proof}

Now we can prove the following theorem.

\begin{theorem}
\label{th3} 
Let $B_q=(S(m,n;q))_{m,n=1}^q$. Then the characteristic polynomial of $B_q$ is
\begin{equation*}
\det(\lambda E_q-B_q)=\lambda^{q-\varphi(q)} (\lambda -q)^{\frac{\varphi(q)+\widetilde{\varphi} (q) }{2} } (\lambda +q)^{\frac{\varphi(q)-\widetilde{\varphi} (q) }{2}}.
\end{equation*}

Especially, the matrix $B_q$ has eigenvalues $0,q,-q $ with multiplicity $q-\varphi(q), \frac{\varphi(q)+\widetilde{\varphi} (q) }{2},\frac{\varphi(q)-\widetilde{\varphi} (q) }{2}$, respectively.
\end{theorem}

\begin{proof} 
The proof proceeds along the same lines as the proof of Theorem \ref{th1}.
We first suppose $\lambda > q$. Since $\det (\exp (M))=\exp( \mathrm{tr} (M))$ holds for any square matrix $M$, we have
\begin{equation*}
\det(\lambda E_q-B_q)= \lambda^q \det( E_q-\frac{1}{\lambda} B_q)= \lambda^q \exp( \mathrm{tr} (\log (E_q-\frac{1}{\lambda} B_q))) .
\end{equation*}
Since
\begin{equation*}
\log (E_q-\frac{1}{\lambda} B_q )= \sum_{j=1}^\infty \frac{(-1)^{j-1}}{j} (-\frac{1}{\lambda} B_q)^j= \sum_{j=1}^\infty \frac{(-1)^{j-1}}{j} (-\frac{1}{\lambda})^j B_q^j
\end{equation*}
holds, we have by Lemma \ref{lem11}
\begin{align}
\mathrm{tr} (\log (E_q-\frac{1}{\lambda} B_q ))=& \sum_{j=1}^\infty \frac{(-1)^{j-1}}{j} (-\frac{1}{\lambda})^j \mathrm{tr} (B_q^j) \nonumber \\
=& \sum_{j=1}^\infty \frac{(-1)^{2j-2}}{2j-1} (-\frac{1}{\lambda})^{2j-1} \mathrm{tr} (B_q^{2j-1}) + \sum_{j=1}^\infty \frac{(-1)^{2j-1}}{2j} (-\frac{1}{\lambda})^{2j} \mathrm{tr} (B_q^{2j}) \nonumber \\
=&  \sum_{j=1}^\infty \frac{1}{2j-1} (-\frac{1}{\lambda})^{2j-1} q^{2j-1} \widetilde{\varphi} (q) + \sum_{j=1}^\infty \frac{-1}{2j} (-\frac{1}{\lambda})^{2j} q^{2j} \varphi(q) \nonumber \\
=& \widetilde{\varphi} (q) \sum_{j=1}^\infty \frac{1}{2j-1} (-\frac{q}{\lambda})^{2j-1}   +\varphi(q) \sum_{j=1}^\infty \frac{-1}{2j} (-\frac{q}{\lambda})^{2j}  \label{eq:th3-01} .
\end{align}

Since it follows that
\begin{align*}
& \sum_{j=1}^\infty \frac{1}{2j-1} x^{2j-1}  = \frac{1}{2} \log \frac{1+x}{1-x}, \\
& \sum_{j=1}^\infty \frac{1}{2j} x^{2j}  = -\frac{1}{2} \log (1-x^2),
\end{align*}
we see that (\ref{eq:th3-01}) is equal to
\begin{equation*}
\frac{\widetilde{\varphi} (q)}{2} \log \frac{1-q/\lambda}{1+q/\lambda}+\frac{\varphi (q)}{2} \log (1-\frac{q^2}{\lambda^2}).
\end{equation*}

Therefore we obtain
\begin{align}
\det(\lambda E_q-B_q)=& \lambda^q \exp \Bigl(\frac{\widetilde{\varphi}(q)}{2} \log \frac{\lambda-q}{\lambda+q}+\frac{\varphi (q)}{2} \log (1-\frac{q^2}{\lambda^2}) \Bigr) \nonumber \\
 =&  \lambda^q \Bigl( \frac{\lambda-q}{\lambda +q} \Bigr)^{\frac{\widetilde{\varphi} (q)}{2}} \Bigl( 1-\frac{q^2}{\lambda^2} \Bigr)^{\frac{\varphi (q)}{2}}  \nonumber \\
 =& \lambda^{q-\varphi(q)} (\lambda -q)^{\frac{\varphi(q)+\widetilde{\varphi} (q) }{2} } (\lambda +q)^{\frac{\varphi(q)-\widetilde{\varphi} (q) }{2} }. \label{eq:th3-1}
\end{align}

Since (\ref{eq:th3-1}) holds for any $\lambda > q$ and both sides of (\ref{eq:th3-1}) are polynomials of $\lambda$, (\ref{eq:th3-1}) holds for any $\lambda$.
This completes the proof of Theorem \ref{th3}.
\end{proof}

Although $B_q$ is not a circulant matrix, we can obtain the eigenvectors of $B_q$ by mimicking the method which was used in the case of $A_q$. Let 
\begin{equation*}
u_j=   (\omega_j,\omega_j^2, \cdots, \omega_j^{q}), 
\end{equation*}
where $\omega_j=\exp (2 \pi i j / q )$ and $1 \leq j \leq q$. We have the following corollary concerning the eigenvectors of $B_q$.

\begin{corollary}
\label{cor3}
Let $B_q=(S(m,n;q))_{m,n=1}^q$. Then the following holds.

If $\gcd(j,q)>1$, then
\begin{align*}
B_q u_j=0.
\end{align*}

If $\gcd(j,q)=1$ and $j+j^* \equiv 0 \Mod{q}$, then
\begin{align*}
B_q u_j=q u_j.
\end{align*}

If $\gcd(j,q)=1$ and $j+j^* \not \equiv 0 \Mod{q}$, then
\begin{align*}
& B_q (u_j+u_{-j^*})=q (u_j+u_{-j^*}), \\
& B_q (u_j-u_{-j^*})=-q (u_j-u_{-j^*}).
\end{align*}
\end{corollary}

\begin{proof}
We have for $1 \leq m \leq q$
\begin{align}
(B_q u_j)(m)=&\sum_{n=1}^q S(m,n;q) u_j(n) \nonumber \\
=& \sum_{n=1}^q \sum_{\substack{k=1 \\ \gcd(k,q)=1}}^q \exp \Bigl(\frac{2 \pi i }{q} (mk+nk^* ) \Bigr) \exp \Bigl(\frac{2 \pi i jn}{q} \Bigr) \nonumber \\
=& \sum_{\substack{k=1 \\ \gcd(k,q)=1}}^q \exp \Bigl(\frac{2 \pi i mk}{q} \Bigr) \sum_{n=1}^q \exp \Bigl(\frac{2 \pi i n}{q} (k^*+j) \Bigr). \label{eq:cor3-1}
\end{align}

Noting that
\begin{equation*}
\sum_{n=1}^q \exp \Bigl(\frac{2 \pi i n}{q} (k^*+j) \Bigr)=\begin{cases} q, & \text{if $k^*+j \equiv 0 \Mod{q}$} ; \\ 0, & \text{otherwise}, \end{cases}
\end{equation*}
we see that (\ref{eq:cor3-1}) is equal to $0$ in the case $\gcd(j,q)>1$ since there exist no $k$ such that $\gcd(k,q)=1$ and $k^*+j \equiv 0 \Mod{q}$ in this case. If $\gcd(j,q)=1$, then (\ref{eq:cor3-1}) is equal to
\begin{align*}
& q \sum_{\substack{k=1 \\ \gcd(k,q)=1 \\ k^*+j \equiv 0 \Mod{q}}}^q \exp \Bigl(\frac{2 \pi i mk}{q} \Bigr) =q \sum_{\substack{k=1 \\ \gcd(k,q)=1 \\ k \equiv -j^* \Mod{q}}}^q \exp \Bigl(\frac{2 \pi i mk}{q} \Bigr)  \\
= &  q \exp \Bigl(-\frac{2 \pi i m j^*}{q} \Bigr) = \begin{cases} q \exp (2 \pi i mj / q  )  , & \text{if $j+j^* \equiv 0 \Mod{q}$} ; \\ q \exp (-2 \pi i m j^* / q ) , & \text{if $j+j^* \not \equiv 0 \Mod{q}$}. \end{cases}
\end{align*}

From this we have for $j$ satisfying $\gcd(j,q)=1$
\begin{equation}
(B_q u_j)=\begin{cases} q u_j , & \text{if $j+j^* \equiv 0 \Mod{q}$} ; \\ q u_{-j^*} , & \text{if $j+j^* \not \equiv 0 \Mod{q}$}. \end{cases} \label{eq:cor3-2}
\end{equation}

Thus $B_q u_j =q u_j$ holds if $\gcd(j,q)=1$ and $j+j^* \equiv 0 \Mod{q}$.

Similarly, if $\gcd(j,q)=1$ and $j+j^* \not \equiv 0 \Mod{q}$, then we obtain $ B_q (u_j+u_{-j^*})=q (u_j+u_{-j^*})$ and $B_q (u_j-u_{-j^*})=-q (u_j-u_{-j^*})$  since $B_q u_j= q u_{-j^*}$ and $B_q u_{-j^*}= q u_j$ hold by (\ref{eq:cor3-2}).
This completes the proof of Corollary \ref{cor3}.
\end{proof} 

\begin{remark}
We note that
\begin{align*}
& \# \{ 1 \leq j \leq q; \ (j,q)=1 , \ \ j+j^* \equiv 0 \Mod{q} \}= \widetilde{\varphi} (q), \\
& \# \{ 1 \leq j \leq q; \ (j,q)=1 , \ \ j+j^* \not \equiv 0 \Mod{q} \}= \varphi(q)-\widetilde{\varphi} (q).
\end{align*}

From this we have
\begin{align*}
& \# \{ 1 \leq j \leq q; \ B_q u_j=q u_j \quad \text{or} \quad  B_q (u_j+u_{-j^*})=q (u_j+u_{-j^*}) \}= \frac{\varphi(q)+\widetilde{\varphi} (q)}{2}, \\
& \# \{ 1 \leq j \leq q; \ B_q (u_j-u_{-j^*})=-q (u_j-u_{-j^*}) \}=\frac{\varphi(q)-\widetilde{\varphi} (q)}{2} ,
\end{align*}
which coincide with the multiplicities of non-zero eigenvalues described in Theorem \ref{th3}.
\end{remark}

We give the following example.

\begin{example}
$B_2= \left( \begin{array}{rr} 1 & -1    \\ -1 & 1  \\ \end{array}  \right)$ has eigenvalues $0,2$. \\
$B_3= \left( \begin{array}{rrr} -1 & 2 & -1  \\ 2& -1 &-1  \\ -1 & -1 & 2  \\ \end{array}  \right)$ has eigenvalues $0,3,-3$. \\
$B_4= \left( \begin{array}{rrrr} -2 & 0 & 2 & 0  \\ 0 & 2 & 0 & -2 \\ 2 & 0 & -2 & 0 \\ 0 & -2 & 0 & 2  \\ \end{array}  \right)$ has eigenvalues $0,4,-4$ with multiplicity $2,1,1$, respectively, since $\varphi (4)=2$ and $\widetilde{\varphi} (4)=0$.  \\
$B_5=\left( \begin{array}{rrrrr} (3-\sqrt{5})/2 & -1-\sqrt{5} & -1+\sqrt{5} & (3+\sqrt{5})/2 & -1 \\ -1-\sqrt{5} & (3+\sqrt{5})/2  & (3-\sqrt{5})/2  &  -1+\sqrt{5} & -1 \\ -1+\sqrt{5} & (3-\sqrt{5})/2 & (3+\sqrt{5})/2 & -1-\sqrt{5} & -1  \\ (3+\sqrt{5})/2 & -1+\sqrt{5} & -1-\sqrt{5} & (3-\sqrt{5})/2 & -1  \\ -1 & -1 & -1 & -1 & 4 \end{array}  \right)$ \\
has eigenvalues $0,5,-5$ with multiplicity $1,3,1$, respectively, since $\varphi (5)=4$ and $\widetilde{\varphi} (5)=2$. We note that $\sqrt{5}$ appears in $B_5$ since $\cos \pi/5=(1+\sqrt{5})/4$. 

More generally, if $p$ is an odd prime number satisfying $p \equiv 1 \Mod 4$, then \\
$B_p$ has eigenvalues $0,p,-p$ with multiplicity $1,\frac{p+1}{2},\frac{p-3}{2}$, respectively, since $\varphi (p)=p-1$ and $\widetilde{\varphi} (p)=2$.

If $p$ is an odd prime number satisfying $p \equiv -1 \Mod 4$, then \\
$B_p$ has eigenvalues $0,p,-p$ with multiplicity $1,\frac{p-1}{2},\frac{p-1}{2}$, respectively, since $\varphi (p)=p-1$ and $\widetilde{\varphi} (p)=0$.

We remark that by Theorem \ref{th1}, if $p$ is any prime number, then $A_p$ has eigenvalues $0,p$ with multiplicity $1,p-1$, respectively, since $\varphi(p)=p-1$.
\end{example}

Next we consider the $ x \times x$ matrix $Y=(Y_{mn})_{m,n=1}^x=(\sum_{q=1}^Q S(m,n:q))_{m,n=1}^x$ where $Q$ is a fixed positive integer and $x$ is the least common multiple of $1,2, \cdots, Q$.
We proceed along similar lines to the case of $X$.

\begin{lemma} 
\label{lem12}
Let $Y=(Y_{mn})_{m,n=1}^x=(\sum_{q=1}^Q S(m,n;q))_{m,n=1}^x$ where $x$ is the least common multiple of $1,2, \cdots, Q$. Then 
\begin{equation*}
 Y^2=xX .
\end{equation*}
\end{lemma}

\begin{proof}
By Lemma \ref{lem1} we see that the $(m,n)$ component of $Y^2$ equals
\begin{align*}
(Y^2)_{mn}= & \sum_{a=1}^x Y_{m a} Y_{a n} =\sum_{a=1}^x (\sum_{q=1}^Q S(m,a;q)) (\sum_{r=1}^Q S(a,n;r)) \\
 =& \sum_{q=1}^Q \sum_{r=1}^Q \sum_{a=1}^x S(m,a;q) S(a,n;r) = \sum_{q=1}^Q \sum_{r=1}^Q x \delta (q,r)  c_q (m-n) \\
= & x \sum_{q=1}^Q  c_q(m-n)=x X_{mn},
\end{align*}
which is, of course, equal to the $(m,n)$ component of $xX$.
\end{proof}

Next we consider the trace of $Y^j$ where $j \in \mathbb{N}$. 
We set $\widetilde{\Phi}(Q)=\sum_{q=1}^Q \widetilde{\varphi}(q)$.

\begin{lemma}
\label{lem13}
Let $Y=(Y_{mn})_{m,n=1}^x=(\sum_{q=1}^Q S(m,n;q))_{m,n=1}^x$ where $x$ is the least common multiple of $1,2, \cdots, Q$. 
For $j \in \mathbb{N}$, we have
\begin{align}
& \mathrm{tr} (Y^{2j})  =x^{2j} \Phi (Q), \label{eq:lem13-1} \\
& \mathrm{tr} (Y^{2j-1}) =x^{2j-1} \widetilde{\Phi} (Q). \label{eq:lem13-2}
\end{align}
\end{lemma}

\begin{proof}
We first prove $\mathrm{tr} (Y) =x \widetilde{\Phi} (Q)$. Since
\begin{align}
\mathrm{tr} (Y) = & \sum_{m=1}^x \sum_{q=1}^Q S(m,m;q) = \sum_{m=1}^x \sum_{q=1}^Q \sum_{\substack{k=1 \\ \gcd(k,q)=1}}^q \exp (\frac{2 \pi i }{q} (mk+mk^*)) \nonumber \\
= & \sum_{q=1}^Q \sum_{\substack{k=1 \\ \gcd(k,q)=1}}^q \sum_{m=1}^x  \exp (\frac{2 \pi i }{q} (k+k^*)m) , \label{eq:lem13-3}
\end{align}
and
\begin{equation*}
\sum_{m=1}^x  \exp (\frac{2 \pi i }{q} (k+k^*)m) =\begin{cases} x, & \text{if $k+k^* \equiv 0 \Mod q $}; \\ 0, & \text{otherwise} \end{cases}
\end{equation*}
hold, we see that $(\ref{eq:lem13-3})$ is equal to
\begin{align*}
& x \sum_{q=1}^Q \sum_{\substack{k=1 \\ \gcd(k,q)=1 \\ k+k^* \equiv 0 \Mod q}}^q 1 \\
=& x \sum_{q=1}^Q \ \# \{ 1 \leq k \leq q; \gcd(k,q)=1 , \ \ k+k^* \equiv 0 \Mod q \} \\
=& x \sum_{q=1}^Q \widetilde{\varphi} (q)= x \widetilde{\Phi } (Q).
\end{align*}

Therefore $\mathrm{tr} (Y)  =x \widetilde{\Phi} (Q)$ holds. \\

Next we prove (\ref{eq:lem13-1}). By Lemma \ref{lem8} and Lemma \ref{lem12} we have
\begin{equation*}
\mathrm{tr} (Y^{2j}) =\mathrm{tr} ((Y^{2})^j) = \mathrm{tr} ((xX)^j)= x^j \mathrm{tr} ( X^j)= x^j x^j \Phi (Q)=x^{2j} \Phi(Q).
\end{equation*}
Therefore (\ref{eq:lem13-1}) holds. 

Next we prove (\ref{eq:lem13-2}) for $j \geq 2$.
We have by Lemma \ref{lem12}
\begin{align*}
\mathrm{tr} (Y^{2j-1}) & = \mathrm{tr} (Y^{2(j-1)} Y)=\mathrm{tr} ((xX)^{j-1} Y)\\
&= \mathrm{tr} (x^{j-1} x^{j-2} X  Y)= x^{2j-3} \mathrm{tr} (XY).
\end{align*}

Since 
\begin{align*}
\mathrm{tr} (XY)=& \sum_{m=1}^x \sum_{q=1}^Q \sum_{r=1}^Q \sum_{a=1}^x  c_q(m-a) S(a,m;r) \\
=& \sum_{q=1}^Q \sum_{r=1}^Q \sum_{m=1}^x \sum_{a=1}^x c_q(m-a) S(a,m;r) \\
=& \sum_{q=1}^Q \sum_{r=1}^Q \delta (q,r) x^2 \widetilde{\varphi}(q)= x^2 \sum_{q=1}^q \widetilde{\varphi}(q)=x^2 \widetilde{\Phi} (Q)
\end{align*}
holds by Lemma \ref{lem4}, we have
\begin{equation*}
\mathrm{tr} (Y^{2j-1}) =x^{2j-3} x^2 \widetilde{\Phi} (Q)=x^{2j-1}  \widetilde{\Phi} (Q).
\end{equation*}
This completes the proof of Lemma \ref{lem13}.
\end{proof}

Let $E_x$ be the $x \times x$ identity matrix where $x$ is a positive integer.
We prove the following theorem.

\begin{theorem} 
\label{th4}
Let $Y=(Y_{mn})_{m,n=1}^x=(\sum_{q=1}^Q S(m,n;q))_{m,n=1}^x$ where $x$ is the least common multiple of $1,2, \cdots, Q$. Then the characteristic polynomial of $Y$ is
\begin{equation*}
\det(\lambda E_x-Y)=\lambda^{x-\Phi(Q)} (\lambda -x)^{\frac{\Phi(Q)+\widetilde{\Phi}(Q)}{2}} (\lambda +x)^{\frac{\Phi(Q)-\widetilde{\Phi}(Q)}{2}}.
\end{equation*}

Especially, the matrix $Y$ has eigenvalues $0,x,-x $ with multiplicity  $x-\Phi(Q), \frac{\Phi(Q)+\widetilde{\Phi}(Q)}{2}, \frac{\Phi(Q)-\widetilde{\Phi}(Q)}{2}$,  respectively.
\end{theorem}

\begin{proof}
The proof proceeds along the same lines as the proof of Theorem \ref{th3}.
We first suppose $\lambda > x$. Since $\det (\exp (M))=\exp( \mathrm{tr} (M))$ holds for any square matrix $M$, we have
\begin{equation*}
\det(\lambda E_x-Y)= \lambda^x \det( E_x-\frac{1}{\lambda} Y)= \lambda^x \exp( \mathrm{tr} (\log (E_x-\frac{1}{\lambda} Y))).
\end{equation*}

Since
\begin{equation*}
\log (E_x-\frac{1}{\lambda} Y )= \sum_{j=1}^\infty \frac{(-1)^{j-1}}{j} (-\frac{1}{\lambda} Y)^j= \sum_{j=1}^\infty \frac{(-1)^{j-1}}{j} (-\frac{1}{\lambda})^j Y^j
\end{equation*}
holds, we have by Lemma \ref{lem13}
\begin{align*}
\mathrm{tr} (\log (E_x-\frac{1}{\lambda} Y ))=& \sum_{j=1}^\infty \frac{(-1)^{j-1}}{j} (-\frac{1}{\lambda})^j \mathrm{tr} (Y^j) \\
=& \sum_{j=1}^\infty \frac{(-1)^{2j-2}}{2j-1} (-\frac{1}{\lambda})^{2j-1} \mathrm{tr} (Y^{2j-1})+\sum_{j=1}^\infty \frac{(-1)^{2j-1}}{2j} (-\frac{1}{\lambda})^{2j} \mathrm{tr} (Y^{2j}) \\
=& \widetilde{\Phi} (Q) \sum_{j=1}^\infty \frac{1}{2j-1} (-\frac{1}{\lambda})^{2j-1} x^{2j-1}  +\Phi(Q) \sum_{j=1}^\infty \frac{-1}{2j} (-\frac{1}{\lambda})^{2j} x^{2j} \\
=& \frac{\widetilde{\Phi} (Q)}{2} \log \frac{1-x/\lambda}{1+x/\lambda}+\frac{\Phi (Q)}{2} \log (1-\frac{x^2}{\lambda^2}).
\end{align*}

Therefore we have
\begin{align}
\det(\lambda E_x-Y)=& \lambda^x \exp \Bigl(\frac{\widetilde{\Phi} (Q)}{2} \log \frac{\lambda-x}{\lambda+x}+\frac{\Phi (Q)}{2} \log (1-\frac{x^2}{\lambda^2}) \Bigr) \nonumber \\
 =&  \lambda^x \Bigl( \frac{\lambda-x}{\lambda +x} \Bigr)^{\frac{\widetilde{\Phi} (Q)}{2}} \Bigl( 1-\frac{x^2}{\lambda^2} \Bigr)^{\frac{\Phi (Q)}{2}}  \nonumber \\
 =& \lambda^{x-\Phi(Q)} (\lambda -x)^{\frac{\Phi(Q)+\widetilde{\Phi} (Q) }{2} } (\lambda -x)^{\frac{\Phi(Q)-\widetilde{\Phi} (Q) }{2} }. \label{eq:th4-1}
\end{align}

Since (\ref{eq:th4-1}) holds for any $\lambda > x$ and both sides of (\ref{eq:th4-1}) are polynomials of $\lambda$, (\ref{eq:th4-1}) holds for any $\lambda$.
This completes the proof of Theorem \ref{th4}.
\end{proof}

\begin{remark}
It is too difficult for us to obtain the eigenvectors of $Y$.
\end{remark}

\bigskip
\hrule
\bigskip

\end{document}